\documentclass{amsart}

\makeatletter
 \def\LaTeX{\leavevmode L\raise.42ex
   \hbox{\kern-.3em\size{\sf@size}{0pt}\selectfont A}\kern-.15em\TeX}
\makeatother

\newcommand{\BibTeX}{{\rm B\kern-.05em{\sc
i\kern-.025emb}\kern-.08em\TeX}}
\newtheorem{col}{Corollary}[section]
\newtheorem{thm}{Theorem}[section]
\newtheorem{lem}[thm]{Lemma}

\theoremstyle{definition}
\newtheorem{defn}{Definition}

\numberwithin{equation}{section}

\begin{document}

\title{Variational splines on Riemannian manifolds with applications to
integral geometry}

\author{Isaac Pesenson}
\address{Department of Mathematics, Temple University,
Philadelphia, PA 19122} \email{pesenson@math.temple.edu}


 \begin{abstract}

 We extend the classical theory of variational interpolating splines to the case
 of  compact Riemannian manifolds.
Our consideration includes in particular such problems as
interpolation of a function by its values on a discrete set of
points and interpolation by
 values of integrals over
  a family of submanifolds.

 The existence and uniqueness of interpolating variational
 spline on a Riemannian manifold is proven.
 Optimal properties of such splines are shown.
 The explicit formulas of
  variational splines in terms of the eigen
 functions of Laplace-Beltrami operator are found.
 It is also shown that in the case of interpolation on discrete sets of points
 variational splines converge to a function in $C^{k}$ norms on manifolds.

Applications of these results to the hemispherical and Radon
transforms on the unit sphere are given.

\end{abstract}

\maketitle

 \section{Introduction and Main Results}

\bigskip

In the present paper we develop variational interpolating splines
in the context of a Riemannian compact manifold with an emphasis
on the so called average splines.

One of the basic examples of such manifolds is the unit sphere
$S^{2}$. The analysis on the two dimensional sphere $S^{2}$ found
many applications in computerized tomography, statistics, signal
analysis, seismology, weather prediction, and computer vision.
During last years the interpolation problem on $S^{2}$, the
problem of evaluating the Fourier
 coefficients of functions on the unit sphere and closely related
problems about quadrature
 formulas on $S^{2}$ attracted interest of many mathematicians.

The theory of variational splines on $R^{d}$ can be found in  [2]
and [7]. The interpolation theory on $S^{2}$ was initiated in [13]
and [16].
 An approximation theory on $S^{2}$  along with many applications and an
 extensive list of
references can be found in the recent monograph [4]. An approach
to interpolation on manifolds, which is  different from our, was
developed in [3].

Our paper was motivated by the following problem which is of
interest for integral geometry.

Let $M, dim M=d,$ be a Riemannian manifold and
 $M_{\nu}, \nu=1,2,...,N,$
 is a family of submanifolds
 $dim M_{\nu}=d_{\nu}, 0\leq d_{\nu}\leq d.$ Given a
set of numbers $v_{1},v_{2},...,v_{N}$ we would like to find a
function for which
\begin{equation}
\int_{M_{\nu}}fdx=v_{\nu}, \nu=1,2,...,N.
\end{equation}

Moreover, we are interested in a "least curved" function that
satisfies the previous conditions. In other words, we seek a
function that satisfies (1.1) and minimizes the functional
$$
u\rightarrow \|(1+\Delta)^{t/2}u\| $$ for appropriate real $t$
where $\Delta$ is the Laplace-Beltrami operator on $M.$ Note that
in the case when the submanifold $M_{\nu}$ is a point the integral
(1.1) is understood as a value of a function at this point.

Our result is that if $s$ is a solution of such variational
problem then the distribution $(1+\Delta)^{t}s$ should satisfy the
following distributional pseudo-differential equation on $M$ for
any $\psi\in C_{0}^{\infty}(M)$,
\begin{equation}
\int_{M}\psi  (1+\Delta)^{t}s
dx=\sum_{\nu=1}^{N}\alpha_{\nu}\int_{M{\nu}}\psi dx,
\end{equation}
where coefficients $\alpha_{\nu}\in \mathbb{C}$ depend just on
$s$.

This equation allows one to obtain the Fourier coefficients of the
function $s$ with respect to eigen functions of the
Laplace-Beltrami operator $\Delta$.

From the very definition our solution $s$ is an "interpolant" in
the sense that it has a prescribed set of integrals. Moreover, we
show that the function $s$ is not just an interpolant but also an
optimal approximation to the set of all functions $f$ in the
Sobolev space $H_{t}(M)$ that satisfy (1.1) and
\begin{equation}
\|(1+\Delta)^{t/2}f\|\leq K,
\end{equation}
for appropriate $K >0.$ Namely, we show that $s$ is the center of
the convex and bounded set of all functions that satisfy (1.1) and
(1.3).

In the Section 4 we develop an approximation theory by
interpolating splines in the case when all submanifolds $M_{\nu}$
are points. We consider convergence of such interpolants in two
different cases. In a first case a set of points
$M_{\nu}=x_{\nu},\nu=1,2,...,N,$ gets denser and in the second one
the order of smoothness of splines goes to infinity. For both
types of convergence we give the rates of convergence of such
interpolants in Sobolev and uniform norms on manifolds.

Our approximation theorem in the case when the order of smoothness
of interpolants goes to infinity but the set of points
$\{x_{\nu}\}$ is fixed, leads to a natural generalization of the
classical Sampling Theorem in the sense that it allows a complete
recovery of band limited functions, i.e. finite linear
combinations of eigen functions of the Laplace-Beltrami operator.

Note that  some generalizations of a classical Sampling Theorem in
the case of periodic functions and periodic splines on the unit
circle can be found in [5], [12] and [14]. In our case the circle
is replaced by a general compact
 Riemannian manifold and trigonometric functions by eigen
 functions of the corresponding Laplace-Beltrami operator.
The paper [1]  contains a Sampling Theorem in the case of $S^{2}$
which is different from our. A result about uniform convergence of
spherical splines to smooth functions
 on spheres was obtained
in [4] with a different rate of convergence and for a different
definition of splines on $S^{2}$ which requires a so called
admissible set of knots.

 In the last section, 5
 we apply our results to the hemispherical and to the Radon
 transforms on the unit sphere $S^{d}$ in $R^{d+1}$.
 Namely, we treat
 the inversion of the hemispherical transform as a specific interpolation
 problem, where information about a function is a set of values of
 integrals over hemispheres from a finite collection of
 hemispheres. A smooth odd function $f$ on $S^{d}$ will be obtained as a uniform
  limit of a sequence of "interpolants" of $f$ which
 have the same integrals over a family of hemispheres as the $f$
 does.

 The same approach we apply to the spherical Radon transform: we
 obtain a smooth function $f$ on $S^{d}$ as the limit of a sequence of
  "interpolants"  of $f$ whose integrals over subspheres from a family
   of subspheres coincide with the integrals of $f$. In both cases our
    interpolants provide an optimal approximations to unknown functions.

Let us give a more detailed account of main results. We consider a
compact Riemannian manifold $M$ and the corresponding
Laplace-Beltrami operator $\Delta$. It is known that the operator
$\Delta$ is elliptic, positive definite and selfadjoint in the
space $L_{2}(M)$ constructed using a Riemannian density $dx$. If $
M$ has a non-empty boundary we assume the Dirichlet boundary
conditions.

The Sobolev space $H_{t}(M), t\in \mathbb{R}$ can be introduced as the
domain of the operator $(1+\Delta)^{t/2}$ with the graph norm

$$
\|f\|_{t}=\|(1+\Delta)^{t/2}f\|, f\in H_{t}(M).
$$

It is known that in this case the operator $\Delta$ has a discrete
spectrum $0=\lambda_{0}<\lambda_{1}\leq \lambda_{2}\leq...,$ and
one can choose corresponding eigen functions $\varphi_{0},
\varphi_{1},...$ which form an orthonormal basis of $L_{2}(M).$ A
distribution $f$ belongs to $H_{t}(M), t\in \mathbb{R},$ if and
only if
$$
\|f\|_{t}=
\left(\sum_{j=0}^{\infty}(1+\lambda_{j})^{t}|c_{j}(f)|^{2}\right)^{1/2}<\infty,
$$
where Fourier coefficients $c_{j}(f)$ of $f$ are given by
$$
c_{j}(f)=<f,\varphi_{j}>=\int_{M}f\overline{\varphi_{j}}dx.
$$

\bigskip

This $L_{2}-$inner product can be also considered as a pairing
between $H_{-t}(M)$ and $H_{t}(M)$ and in this sense every element
of $H_{-t}(M)$ can be identified with a continuous functional on
$H_{t}(M)$.

 Let $F_{\nu}, \nu=1,2,...,N,$ be a set of
distributions from a certain $H_{-t_{0}}(M), t_{0}\geq 0.$ Our
main assumption about the family of functionals $\{F_{\nu}\}$ is
that the functionals $F_{\nu}$ can be "separated" in the following
sense:

\bigskip

 \textbf{Independence Assumption.} \textsl{There are functions
  $\vartheta_{\nu}\in C^{\infty}_{0}(M)$ such
that}

$$
F_{\nu}(\vartheta_{\mu})=\delta_{\nu\mu},
$$
\textsl{where $\delta_{\nu\mu}$ is the Kronecker delta.}

Note, that this assumption implies in particular that the
functionals $F_{\nu}$ are linearly independent. Indeed, if we have
that for certain coefficients $\gamma_{1},\gamma_{2}, ...,
\gamma_{N}$
$$
\sum _{\nu=1}^{N}\gamma_{\nu}F_{\nu}=0,
$$
then for any $1\leq\mu\leq N$
$$
0=\sum_{\nu=1}^{N}\gamma_{\nu}F_{\nu}(\vartheta_{\mu})=\gamma_{\mu}.
$$

The families of distributions that satisfy our condition
  include

  a) Finite families of $\delta$ functionals and their
  derivatives.

  b) Sets of integrals over submanifolds from a finite family of
submanifolds of any codimension.

\textbf{Variational Problem}

 Given a sequence of complex numbers
$v=\{v_{\nu}\}, \nu=1,2,...,N,$ and a $t>t_{0}$ we consider the
following variational problem:

\bigskip

\textsl{Find a function  $u$ from the space $H_{t}(M)$ which has
the following properties:}

1) $ F_{\nu}(u)=v_{\nu}, \nu=1,2,...,N, v=\{v_{\nu}\},$

2) $u$ \textsl{minimizes functional $u\rightarrow \|(1+\Delta)
^{t/2}u\|$.}

\bigskip

We show that the solution to Variational problem does exist and is
unique for any $t>t_{0}$ even without any assumption. But we need
the Independence  Assumption in order to determine the Fourier
coefficients of the solution. The solution to the Variational
Problem will be called a spline and will be denoted as $s_{t}(v).$
 The set of all solutions for a fixed set of distributions
$F=\{F_{\nu}\}$ and a fixed $t$ will be denoted as $S(F,t).$

Given a function $f\in H_{t}(M)$ we will say that the unique
spline $s$ from $S(F,t)$ interpolates $f$ on $F$ if
$$
F_{\nu}(f)=F_{\nu}(s).
$$
Such spline will be denoted as $s_{t}(f).$

From the point of view of the classical theory of variational
 splines it would be more natural to consider minimization of the
functional
$$
u\rightarrow \|\Delta^{t/2}u\|.
$$
However, in the case of a general compact manifolds it is easer to
work with the operator $1+\Delta$ since this operator is
invertible.

Our main result concerning  variational splines is the following.
\begin{thm}
If every functional $F_{\nu}, \nu=1,2,...,N,$ belongs to
$H_{-t_{0}}(M)$, if the Independence and Reality Assumptions are
satisfied and if $t>t_{0}+d/2$, then for any given sequence
$v=\{v_{\nu} \}, \nu=1,2,...N,$ the following statements are
equivalent:

1) $s_{t}(v)$ is the solution to \textsl{the Variational Problem};

2) $s_{t}(v)$ satisfies the following equation in the sense of
distributions
\begin{equation}
(1+\Delta)^{t}s_{t}(v)=\sum_{\nu=1}^{N}\alpha_{\nu}(s_{t}(v))\overline{F_{\nu}},
t>t_{0}+d/2,
\end{equation}
where $\alpha_{1}(s_{t}(v)),...,\alpha_{N}(s_{t}(v))$ form a
solution of the $N\times N$ system
\begin{equation}
\sum_{\nu=1}^{N}\beta_{\nu\mu}\alpha_{\nu}(s_{t}(v))=v_{\mu},
\mu=1,...,N,
\end{equation}
and
\begin{equation}
\beta_{\nu\mu}=\sum_{j=0}^{\infty}(1+\lambda_{j})^{-t}\overline{F_{\nu}(\varphi_{j})}
F_{\mu}(\varphi_{j});
\end{equation}

3) the Fourier series of $s_{t}(v)$ is the following
$$
s_{t}(v)=\sum_{j=0}^{\infty}c_{j}(s_{t}(v))\varphi_{j},
$$
where
$$
c_{j}(s_{t}(v))=<s_{t}(v),\varphi_{j}>=(1+\lambda_{j})^{-t}
\sum_{\nu=1}^{N}\alpha_{\nu}(s_{t}(v))\overline{F_{\nu}(\varphi_{j})}.
$$
\end{thm}

The statement that (1.4) is satisfied in the sense of
distributions means, that for any  $\psi\in H_{t}(M), t\geq
t_{0},$

$$\int_{M}(1+\Delta)^{t}s_{t}(v)\overline{\psi} dx=\sum_{\nu=1}^{N}
\alpha_{\nu}(s_{t}(v))\overline{F_{\nu}(\psi )}.$$

It is important to note that the system (1.5) is always solvable
according to our uniqueness and existence result for the
\textsl{Variational Problem.}

It is also necessary to note that the series (1.6) is absolutely
convergent if $t>t_{0}+d/2$. Indeed, since functionals $F_{\nu}$
are continuous in the Sobolev space $H_{t_{0}}(M)$ we obtain that
for any normalized eigen function $\varphi_{j}$ which corresponds
to the eigen value $\lambda_{j}$ the following inequality holds
true

$$|F_{\nu}(\varphi_{j})|\leq
C(M,F)\|(1+\Delta)^{t_{0}/2}\varphi_{j}\|\leq
C(M,F)(1+\lambda_{j})^{t_{0}/2}.
$$
So
$$
|\overline{F_{\nu}(\varphi_{j})}F_{\mu}(\varphi_{j})| \leq
C(M,F)(1+\lambda_{j})^{t_{0}},
$$
and
$$
|(1+\lambda_{j})^{-t}\overline{F_{\nu}(\varphi_{j})}F_{\mu}(\varphi_{j})|\leq
C(M,F)(1+\lambda_{j})^{(t_{0}-t)}.
$$

It is known that the series
$$
\sum_{j}\lambda_{j}^{-\tau},
$$
which defines the $\zeta-$function of the Laplace-Beltrami
operator, converges if $\tau>d/2$. This implies absolute
convergence of (1.6) in the case $t>t_{0}+d/2$.

We show that for a given function $f\in H_{t}(M)$ its
interpolating spline $s_{t}(f)$ has the following important
property which means that it is always an optimal approximation in
the sense of Golomb and Weinberger [6]. Namely, if  $Q(F,f,t,K)$
is the convex bounded and closed set of all functions $g$ from
$H_{t}(M)$ such that

\bigskip

1) $F_{\nu}(g)=F_{\nu}(f), \nu=1,2,...,N,$

and

2) $\|g\|_{t}\leq K,$ for a real $ K\geq \|s_{t}(f)\|_{t},$

\bigskip
then $s_{t}(f)$ is the center of $Q(F,f,t,K)$. This means that for
any $g\in Q(F,f,t,K)$
\begin{equation}
  \|s_{t}(f)-g\|_{t}\leq \frac{1}{2} diam Q(F,f,t,K).
\end{equation}

To formulate our Approximation Theorem in the case when the
distributions $F_{\nu}$ are Dirac distributions $\delta_{x_{\nu}}$
at points $x_{\nu}$ we need the notion of a $\rho$-lattice.

We will say that a finite set of points $X_{\rho}=\{x_{\nu}\},
\nu=1,2,...,N,$ is a $\rho$-lattice, if

1) the balls $B(x_{\nu},\rho/2)$ are disjoint,

2) the balls $B(x_{\nu},\rho)$  form a cover of $M$.

The set of functionals $F=\{F_{\nu}\}$ associated with a
$\rho$-lattice $X_{\rho}$ is the set of Dirac distributions
$\delta_{\nu}=\delta_{x_{\nu}}, \nu=1,2,...,N$ on
$C_{0}^{\infty}(M)$.

We have the following result about convergence of splines in
uniform spaces $C^{k}(M)$.

\begin{thm}
There exist a constant $\rho(M)>0,$ so that for any $t>d/2+k,
d=dim M,$
 there exists
a constant $C(M,t)$  such that for any $\rho$-lattice $X_{\rho}$
with $\rho<\rho(M),$ and any smooth function $f$ the following
holds true
\begin{equation}
\|s_{2^{m}d+t}(f)(x)-f(x)\|_{C^{k}(M)}\leq
\left(C(M,t)\rho^{2}\right)^{2^{m}d} \|(1+\Delta)^{2^{m}d+t}f\|
\end{equation}
for all $m=0,1,... .$

Moreover, if $f$ is $\omega$-band limited, i.e. belongs to the
$span$ of eigen functions whose eigen values are not greater than
$\omega$, then
\begin{equation}
\|s_{2^{m}d+t}(f)(x)-f(x)\|_{C^{k}(M)}\leq (1+\omega)^{t}
\left(C(M,t)\rho^{2}(1+\omega)\right)^{2^{m}d}\|f\|.
\end{equation}
\end{thm}

The inequality (1.8) shows that convergence in $C^{k}(M)$ takes
place when $\rho$ goes to zero and the index $2^{m}d+t$ is fixed.

The second inequality (1.9) shows that right-hand side goes to
zero for a fixed $\rho$- lattice $X_{\rho}$ as long as
$$
\rho<\left(C(M)(1+\omega)\right)^{-1/2}
$$
and $m$ goes to infinity.

The second statement is, in fact, a generalization of the
classical Sampling Theorem to the case of a general compact
manifold.

In the last Section, 5  we use above results for an approximate
inversion of the hemispherical and Radon transforms on spheres. In
these cases the functionals $F_{\nu}$ are integrals over
hemispheres and subspheres respectively. In these situations by
applying our Approximation Theorem on the "dual" sphere we show
that if the set of hemispheres (resp. subspheres) gets "denser"
than interpolants converge to the original function in
$C^{k}(S^{d})$ norms.

Note that some of these results in the case of the spherical Radon
transform were obtained in [11].

\section{Variational splines on manifolds}

In this section we prove the existence and  uniqueness of the
interpolating variational spline on a compact Riemannian manifold.
We also describe  Fourier coefficients of such splines with
respect to an orthonormal system of eigen functions of the
Laplace-Beltrami operator.

 We assume that $F=\{F_{\nu}\}, \nu=1,...,N,$ is a family
of distributions from a $H_{-t_{0}}(M), t_{0}\geq 0.$

 Given a sequence of complex numbers
$\{v_{\nu}\}, \nu=1,2,...,N,$ and a $t>t_{0}$ we show  that the
corresponding \textsl{Variational Problem} does have a unique
solution.

\begin{thm}
If every functional $F_{\nu}, \nu=1,2,...,N$ belongs to
$H_{-t_{0}}(M), t_{0}\geq 0$, if $t>t_{0}$, then the \textsl{
Variational Problem} does have the unique solution for any
sequence of values $(v_{1},v_{2},... v_{N})$.
\end{thm}

\begin{proof}
Consider the set $V^{0}_{t}(F)\subset H_{t}(M), t>t_{0}, $ of all
functions from $H_{t}(M)$ such that for every $1\leq \nu\leq N,
F_{\nu}(f)=0.$

Given a sequence of complex numbers $(v_{1}, v_{2}, ..., v_{N})$ the
 linear manifold

 $$V_{t}(F,v_{1},...v_{N}), t>t_{0}$$
 of all functions $f$ from
$H_{t}(M)$ such that $F_{\nu}(f)=v_{\nu}, \nu=1,...,N,$ is a
shift of the closed subspace $V^{0}_{t}(F)$, i.e.

$$ V_{t}(F,v_{1},...,v_{N})=V^{0}_{t}(F)+g,$$
where $g$ is any function from $H_{t}(M)$ such that
$F_{\nu}(g)=v_{\nu}, \nu=1,2,...,N.$

Consider the orthogonal projection $g_{0}$ of $g\in H_{t}(M)$
onto the space $V^{0}_{t}(F)$ with respect to the inner product
in $H_{t}(M)$:

$$<f_{1},
f_{2}>_{H_{t}(M)}=<(1+\Delta)^{t/2}f_{1},(1+\Delta)^{t/2}f_{2}>_{L_{2}(M)}=
$$
$$
 \int_{M}(1+\Delta)^{t/2}f_{1}\overline{(1+\Delta)^{t/2}f_{2}}dx.
$$

It is clear that $s_{t}(v)=g-g_{0}\in V_{t}(F,v_{1},...,v_{N})$ is
the unique solution of the Variational Problem. Indeed, to show
that $s_{t}(v)$ minimizes the functional

$$u\rightarrow \|(1+\Delta)^{t/2}u\|$$
on the set $V_{t}(F,v_{1},...,v_{N})$ we note that any function
from $V_{t}(F,v_{1},...,v_{N})$ can be written in the form
$s_{t}(v)+h,$ where $h\in V^{0}_{t}(F)$. For such a function we
have

$$
\|(1+\Delta)^{t/2}(s_{t}(v)+h)\|^{2}=
$$
$$
\|(1+\Delta)^{t/2}s_{t}(v)\|^{2}+2<s_{t}(v),h>_{H_{t}(M)}+
\|(1+\Delta)^{t/2}h\|^{2}.
$$
Since $s_{t}(v)=g-g_{0}$ is orthogonal to $V^{0}_{t}(F)$ we obtain

$$
\|(1+\Delta)^{t/2}(s_{t}(v)+\sigma h)\|^{2}=
\|(1+\Delta)^{t/2}s_{t}(v)\|^{2}+
|\sigma|^{2}\|(1+\Delta)^{t/2}h\|^{2}, h\in V^{0}_{t}(F),
$$
that shows that the function $s_{t}(v)$ is the minimizer.
\end{proof}

\begin{col}
A function $u\in H_{t}(M)$ is a solution of the Variational
Problem if and only if it is orthogonal to the subspace
$V^{0}_{t}(F)$ and $F_{\nu}(u)=v_{\nu}, \nu=1,2,... .$
\end{col}

The next Theorem gives the characteristic property of splines.

\begin{thm}
If in addition to conditions of the Theorem 2.1 distributions
$F_{1},...,F_{N}$ satisfy \textsl{Independence  Assumption} then a
function $s_{t}(v)\in H_{t}(M),t>t_{0}$ is a solution of
\textsl{The Variational Problem} if and only if it satisfies the
following equation in the sense of distributions
\begin{equation}
(1+\Delta)^{t}s_{t}(v)=\sum_{\nu=1}^{N}
\alpha_{\nu}(s_{t}(v))\overline{F_{\nu}}.
\end{equation}

In other words, for any  smooth $\psi$

$$<(1+\Delta)^{t}s_{t}(v),\psi>_{L_{2}(M)}=\sum_{\nu=1}^{N}
\alpha_{\nu}(s_{t}(v))\overline{F_{\nu}(\psi )}.$$

\end{thm}
\begin{proof}
We already know that every solution of \textsl{The Variational
Problem} is orthogonal to $V^{0}_{t}(F)$ in the Hilbert
space $H_{t}(M)$ i.e. for any $h\in V^{0}_{t}(F)$
\begin{equation}
0=<s_{t}(v),h>_{H_{t}(M)}
=\int_{M}(1+\Delta)^{t/2}s_{t}(v)\overline{(1+\Delta)^{t/2}h}.
\end{equation}

According to our \textbf{ Independence Assumption } there exist a
set of functions
 $\vartheta=\{\vartheta_{\nu}\}, \nu=1,2,...,N,$ from $C_{0}^{\infty}(M)$
such that $F_{\mu}(\vartheta_{\nu})=\delta_{\mu\nu}$, where
$\delta_{\nu\mu}$ is the Kronecker delta. Then for any $\psi\in
C_{0}^{\infty}(M)$ the function

$$\psi-\sum_{\nu=1}^{N}F_{\nu}(\psi)\vartheta_{\nu}
$$
belongs to $V^{0}_{t}(F)$ and because of (2.2)

$$
0=<s_{t}(v),\psi-\sum_{\nu=1}^{N}F_{\nu}(\psi)\vartheta_{\nu}>_{H_{t}(M)}=
\int_{M}(1+\Delta)^{t/2}s_{t}(v)
\overline{(1+\Delta)^{t/2}(\psi-\sum_{\nu=1}^{N}F_{\nu}(\psi)\vartheta_{\nu})}=
$$
$$
\int_{M} (1+\Delta)^{t}s_{t}(v)\left(\overline{\psi-
\sum_{\nu=1}^{N}F_{\nu}(\psi)\vartheta_{\nu}}\right).
$$
In other words,

$$
\int_{M}(1+\Delta)^{t}s_{t}(v)\overline{\psi}=\sum
_{\nu=1}^{N}\overline{F_{\nu}(\psi)}\int_{M}(1+\Delta)
^{t}s_{t}(v)\overline{\vartheta_{\nu}}dx.
$$

If we set
$$
\alpha_{\nu}(s_{t}(v),\vartheta)=\int_{M}(1+\Delta)^{t}s_{t}(v)\overline{\vartheta_{\nu}}dx,
$$
we obtain  that $(1+\Delta)^{t}s_{t}(v)$ is a distribution of the
form

$$
(1+\Delta)^{t}s_{t}(v)=\sum_{\nu=1}^{N}\alpha_{\nu}(s_{t}(v),\vartheta)
\overline{F_{\nu}},
$$
where
$$
\overline{F_{\nu}}(\psi)=\overline{F_{\nu}(\psi)}.
$$
 So
every solution of the variational problem is a solution of (2.1).

Note, that if $\zeta=\{\zeta_{\nu}\}$ is another
$C_{0}^{\infty}(M)-$family for which
$F_{\nu}(\zeta_{\mu})=\delta_{\nu\mu}$, then we have the identity
$$
\sum_{\nu=1}^{N}(\alpha_{\nu}(s_{t}(v),\vartheta)-\alpha_{\nu}(s_{t}(v),\zeta))F_{\nu}=0
$$
which implies  that
$$
\alpha_{\nu}(s_{t}(v),\vartheta)-\alpha_{\nu}(s_{t}(v),\zeta)=0
$$
i.e. coefficients
$\alpha_{\nu}(s_{t}(v),\vartheta)=\alpha_{\nu}(s_{t}(v))$ are
independent of the choice of the family of functions $\vartheta$.

Conversely, if $u$ is a solution of (2.1) then since $F_{\nu}$
belongs to the space $H_{t_{0}}(M),$ and $t>t_{0}\geq 0,$ the
Regularity Theorem for elliptic operator $(1+\Delta)^{t}$  implies
that $u\in H_{-t_{0}+ 2t}(M)\subset H_{t}(M)$
 and for any $h\in V^{0}_{t}(F)$

$$
<u,h>_{H_{t}(M)}=<(1+\Delta)^{t/2}u,(1+\Delta)^{t/2}h>=<(1+\Delta)^{t}u,h>=
\sum_{\nu=1}^{N}\alpha_{\nu}(u)F_{\nu}(h)=0,
$$
that shows that $u$ is a the solution of \textsl{The Variational
Problem}.
\end{proof}

 As a consequence of the Theorem we obtain the fact that the set
of all solutions of \textsl{The Variational Problem} is linear. In
particular, every spline $s_{t}(v)\in S(F,t)$ has the following
representation through its values $F_{\nu}(s_{t}(v))=v_{\nu},
\nu=1,...,N,$
 on $X$:
\begin{equation}
s_{t}(v)=\sum_{\nu=1}^{N}v_{\nu}l^{\nu},
\end{equation}
where $F_{\nu}(s_{t}(v))=v_{\nu},$ and $l^{\nu}\in S(F,t),
\nu=1,2,... ,N, $ is so called Lagrangian spline that defined by
conditions $F_{\mu}(l^{\nu})=\delta_{\nu\mu}, \mu=1,2,...,N.$

To obtain another representation of splines we will need the
solutions $E_{\nu}^{t}, $ of the following distributional equations
\begin{equation}
(1+\Delta)^{t}E_{\nu}^{t}=\overline{F_{\nu}}.
\end{equation}

To find $E_{\nu}^{t}$ we note that in the sense of distributions

\begin{equation}
\overline{F_{\nu}}=\sum_{j=0}^{\infty}
\overline{F_{\nu}(\varphi_{j})}\varphi_{j}
\end{equation}
which shows that
\begin{equation}
E_{\nu}^{t}=\sum_{j=0}^{\infty}(1+\lambda_{j})^{-t}
 \overline{F_{\nu}(\varphi_{j})}\varphi_{j}.
\end{equation}

According to the Theorem 2.2 every spline $s_{t}(v)$ is a solution
of

$$
(1+\Delta)^{t}s_{t}(v)=
\sum_{\nu=1}^{N}\alpha_{\nu}(s_{t}(v))\overline{F_{\nu}},
$$
and along with (2.4) it gives
$$
(1+\Delta)^{t}s_{t}(v)=\sum _{\nu}\alpha_{\nu}(s_{t}(v))
 \sum_{j}\overline{F_{\nu}(\varphi_{j})}\varphi_{j},
$$
that implies the following representation
\begin{equation}
s_{t}(v)=\sum_{\nu=1}^{N}\alpha_{\nu}(s_{t}(v))E_{\nu}^{t}.
\end{equation}

Note that so far we have used just the assumption that $t>t_{0}$.
To get more information about $s_{t}(v)$ we will need a stronger
assumption that $t>t_{0}+d/2$.

 The next Theorem shows how to find an explicit connection
between a sequence of values $(v_{1},...,v_{N})$ and corresponding
sequence $(\alpha_{1}(s_{t}(v)),...,\alpha_{N}(s_{t}(v)))$ in the
case when $t>t_{0}+d/2$. We also assume that Independence
 Assumption is satisfied.

\begin{thm}
If $0=\lambda_{0}<\lambda_{2}\leq.....$ is the sequence of eigen
values of $\Delta $ and $\varphi_{0}, \varphi_{1},...$ is the
corresponding sequence of orthonormal eigen functions, then for
any spline $s_{t}(v)\in S(F,t),$ such that $ F=\{F_{\nu}\},$
$F_{\nu}\in H_{-t_{0}}(M), t>t_{0}+d/2,
F_{\nu}(s_{t}(v))=v_{\nu},\nu=1,...,N,$
 the vector
$\alpha(s_{t}(v))=(\alpha_{1}(s_{t}(v)),...,\alpha_{N}(s_{t}(v)))$
is the solution of the following $N\times N$ system
\begin{equation}
\sum_{\nu=1}^{N}\beta_{\nu\mu}\alpha_{\nu}(s_{t}(v))=v_{\mu},
\mu=1,...,N,
\end{equation}
where
\begin{equation}
\beta_{\nu\mu}=\sum_{j=0}^{\infty}(1+\lambda_{j})^{-t}
\overline{F_{\nu}(\varphi_{j})}F_{\mu}(\varphi_{j}).
\end{equation}
\end{thm}
\begin{proof}
From (2.6)
$$
c_{j}(E^{t}_{\nu})=(1+\lambda_{j})^{-t}\overline{F_{\nu}(\varphi_{j})},
$$
or
$$
E_{\nu}^{t}=\sum_{j=0}^{\infty}(1+\lambda_{j})^{-t}
\overline{F_{\nu}(\varphi_{j})}\varphi_{j}.
$$
Combining this formula with (2.7) we obtain
$$
v_{\mu}=F_{\mu}(s_{t}(v))=\sum_{\nu=1}^{N}\alpha_{\nu}F_{\mu}(E_{\nu}^{t})=
$$
$$
\sum_{\nu=1}^{N}\alpha_{\nu}(s_{t}(v))\beta_{\nu\mu},
$$
where
$$
\beta_{\nu\mu}=\sum_{j=0}^{\infty}(1+\lambda_{j})^{-t}
\overline{F_{\nu}(\varphi_{j})}F_{\mu}(\varphi_{j}).
$$

\end{proof}

Note that according to the existence and uniqueness result the
system (2.8) is \textsl{always } solvable.

It was explained in the introduction that the series (2.9) is
absolutely convergent if $t>t_{0}+d/2$.

The following Theorem gives the Fourier coefficients of splines.
\begin{thm}
If the Independence Assumption is satisfied and $t>t_{0}+d/2$,
then the Fourier coefficients of the spline $s_{t}(v)\in S(F,t)$
are given by the following formulas
$$
c_{j}(s_{t}(v))=<s_{t}(v),\varphi_{j}>=
(1+\lambda_{j})^{-t}\sum_{\nu=1}^{N}
\alpha_{\nu}(s_{t}(v))\overline{F_{\nu}(\varphi_{j})}, j=0,1, ...
$$
where the vector  $(\alpha_{1}(s_{t}(v)),
...,\alpha_{N}(s_{t}(v)))$ is the solution of the corresponding
system (2.8).
\end{thm}
\begin{proof}
If
$$
s_{t}(v)=\sum_{j=0}^{\infty}c_{j}(s_{t}(v))\varphi_{j}
$$
then
$$
(1+\Delta)^{t}s_{t}(v)=\sum_{j=0}^{\infty}(1+\lambda_{j})^{t}c_{j}(s_{t}(v))\varphi_{j},
$$
and at the same time by (2.1)
$$
(1+\Delta)^{t}s_{t}(v)=\sum_{\nu=1}^{N}\alpha_{\nu}(s_{t}(v))\overline{F_{\nu}}.
$$

Combining last two formulas with (2.5) we obtain
$$
c_{j}(s_{t}(v))=\sum_{j}\left(\sum_{\nu}\alpha_{\nu}(s_{t}(v))
\overline{F_{\nu}(\varphi_{j})}\right)\varphi_{j}.
$$
The Theorem 2.4 is proved.

\end{proof}

\section{Another extremal property of variational splines}

The goal of the section is to show that variational splines
provide an approximation which is optimal.

Recall that for a given family of distributions $F=\{F_{\nu}\},
 \nu=1,2,...,N,$ from a $H_{-t_{0}}(M),t_{0}\geq 0,$ and
 a set of complex numbers $v_{1},v_{2},...,v_{N}$
 the notation
 $V_{t}(F;v_{1},...v_{N})$ means the linear manifold of all
 functions $f$ from $H_{t}(M), t>t_{0},$ such that
 $F_{\nu}(f)=v_{\nu}, \nu=1,2,...,N.$
\begin{lem}
For any $g\in V_{t}(F;v_{1},...v_{N})$
$$
\|g\|_{t}\geq
\left(\sum_{\nu=1}^{N}v_{\nu}\alpha_{\nu}(s_{t}(v))\right)^{1/2}
$$
where $s_{t}(v)$ is the unique spline from $
V_{t}(F;v_{1},...v_{N}).$
\end{lem}
\begin{proof}
Let us note that the distance from zero to the subspace $V_{t}(F,
v_{1},...,v_{N})$ in the metric of the space $H_{t}(M)$ is exactly
the Sobolev norm of the spline  $s_{t}(v)\in V_{t}(F, v_{1},
...,v_{N}).$
 This norm can be expressed in
terms of the sequences $v_{1}, ... ,v_{N}$ and
$\alpha_{1}(s_{t}(v)), ... , \alpha_{N}(s_{t}(v))$. Indeed,
$$
\|s_{t}(v)\|_{H_{t}(M)}=<(1+\Delta)^{t/2}s_{t}(v),(1+\Delta)^{t/2}s_{t}(v)>^{1/2}=
<(1+\Delta)^{t}s_{t}(v),s_{t}(v)>^{1/2}=
$$
$$
<\sum_{\nu=1}^{N}\alpha_{\nu}(s_{t}(v))F_{\nu},s_{t}(v)>^{1/2}=
\left(\sum_{\nu=1}^{N}\alpha_{\nu}(s_{t}(v))v_{\nu}\right)^{1/2}.
$$

The Lemma is proved.
\end{proof}
 It other words  the intersection
$$
Q(F,f,t,K)=V_{t}(F, v_{1}, ..., v_{N})\bigcap B_{t}(K),
$$
where $B_{t}(K)$ is the ball in $H_{t}(M)$ of radius $K$, is not
empty if and only if

$$
K\geq\|s_{t}(v)\|_{H_{t}(M)}=\left(\sum_{\nu=1}^{N}\alpha_{\nu}(s_{t}(v))v_{\nu}\right)^{1/2}.
$$

 By the very definition spline
$s_{t}(v)$ is a solutions of an optimization problem. Now we prove
another extremal property of $s_{t}(v)$.

\begin{lem}
The function  $s_{t}(v)$ is the center of the convex, closed and
bounded set  $Q(F, v, t, K)$ for any $K\geq
\|s_{t}(v)\|_{H_{t}(M)}$ .
\end{lem}
\begin{proof}

We will show that if
$$s_{t}(v)+h\in Q(F,v,t,K)
$$
for some function $h$ from the Sobolev space $H_{t}(M)$ then the
function $s_{t}(v)-h$  also belongs to the same intersection.
 Indeed the last assumption
shows that for any $\nu$, $F_{\nu}(h)=0$  and then by (2.2)

$$\int_{M}(1+ \Delta)^{t/2}s_{t}(v)\overline{(1+\Delta)^{t/2}h}=0.$$

But then

$$\|(1+\Delta)^{t/2}(s_{t}(v)+h)\|=\|(1+\Delta)^{t/2}(s_{t}(v)-h)\|.$$

In other words,
 $$\|(1+\Delta)^{t/2}(s_{t}(v)-h)\|\leq K $$
 and because $F_{\nu}(s_{t}(v)+h)=F_{\nu}(s_{t}(v)-h)$ for any $\nu=1,2,...,N,$
  the function $s_{t}(v)-h$ belongs to
 $Q(F,v,t,K).$

 The Lemma is proved.
\end{proof}

\begin{col}
For any $g\in Q(F,v,t,K)$ the following inequality holds true
$$
\|s_{t}(v)-g\|_{t}\leq\frac{1}{2}diam Q(F,v,t,K).
$$
\end{col}

\section{Spline approximation on manifolds}

In this section we develop the approximation theory by variational
splines in the case when the set of distributions $F_{i}$ is a set
of  delta functions on  certain set of points of $M$.
\begin{defn}
We will say that a finite set of points $X=(x_{1},...,x_{N})$ is a
$\rho$-lattice, if

1) The balls $B(x_{i}, \rho/2)$ are disjoint.

2) The balls $B(x_{i}, \rho)$ form a cover of $M$.
\end{defn}
We will  need the following result from [8], [9].

\begin{lem}
There exist constants $C(M)>0, \rho(M)>0$
 such that for any $\rho<\rho(M)$, any $\rho$-lattice
  $X_{\rho}=\{x_{i}\}$ and for any $f\in H_{2d}(M)$ such that
  $f(x_{\nu})=0$ for all $ x_{\nu}\in X_{\rho},$ the following inequality holds true

$$
\|f\|\leq C(M)\rho^{2d}\|(1+\Delta)^{d}f\|, d=\dim M.
$$
\end{lem}

The next goal is to extend the last estimate to higher Sobolev
norms.

\begin{thm}
There exist constants $C(M)>0,\rho(M)>0,$ such that for any
$0<\rho<\rho(M)$, any $\rho$-lattice $X_{\rho}=\{x_{\nu}\}$,
  any smooth $f$ which is zero on
$X_{\rho}$ and any $t\geq 0$

$$
\|(1+\Delta)^{t}f\|\leq \left(C(M)\rho^{2d}\right)^{2^{m}}
\|(1+\Delta)^{2^{m}d+t} f\|, t\geq 0
$$
for all $m=0,1,... .$

\end{thm}

We will obtain this estimate as a consequence of the following
Lemma which is similar to a Lemma from [10].

\begin{lem}
If for some $f\in H^{2s}(M), a,s>0,$
\begin{equation}
\|f\|\leq a\|\Delta^{s}f\|,
\end{equation}
then for the same $f, a, s$ and all $t\geq 0, m=2^{l}, l=0, 1,
...,$
\begin{equation}
\|\Delta^{t}f\|\leq a^{m}\|\Delta^{ms+t}f\|,
\end{equation}
if $f\in H^{2(ms+t)}(M).$

\end{lem}

\begin{proof}
Let us remind the reader that $\{\lambda_{j}\}$ is the set of
eigen values of the operator $\Delta$ and $\{\varphi_{j}\}$ is the
set of corresponding orthonormal eigen functions. Let
$\{c_{j}=<f,\varphi_{j}>\}$ be the set of Fourier
 coefficients of the function
$f$ with respect to the orthonormal basis $\{\varphi_{j}\}.$ Using
the Plancherel Theorem we can write our assumption (4.1) in the
form

$$
\|f\|^{2}\leq a^{2}\left(\sum_{\lambda_{j}\leq a^{-1/s}}
\lambda_{j}^{2s}|c_{j} |^{2}+\sum_{\lambda_{j}>
a^{-1/s}}\lambda_{j}^{2s}|c_{j} |^{2}\right).
$$
Since for the first sum $a^{2}\lambda_{j}^{2s}\leq 1$,

$$0\leq\sum_{\lambda_{j}\leq a^{-1/s}}(|c_{j}|^{2}-a^{2}\lambda_{j}^{2s}
|c_{j}|^{2})\leq \sum_{\lambda_{j}>
a^{-1/s}}(a^{2}\lambda_{j}^{2s} |c_{j}|^{2}-|c_{j}|^{2}).
$$

Multiplication of  this inequality by $a^{2}\lambda_{j}^{2s}$
 will only
 improve the existing inequality
and then using the Plancherel Theorem once again we will obtain

$$\|f\|\leq a\|\Delta^{s}f\|\leq a^{2}\|\Delta^{2s}f\|.
$$

It is now clear that using induction we can prove

$$
\|f\|\leq a^{m}\|\Delta^{ms}f\|, m=2^{l}, l\in \mathbb{N}.
$$

But then, using the same arguments we have for any $\tau>0$

$$
0\leq \sum_{\lambda_{j}\leq a^{-1/s}}(a^{2\tau}\lambda_{j}^{2\tau
s} |c_{j}|^{2}
-a^{2(m+\tau)}\lambda_{j}^{2(m+\tau)s}|c_{j}|^{2})\leq
$$
$$\sum_{\lambda_{j}> a^{-1/s}}(a^{2(m+\tau)}\lambda_{j}^{2(m+\tau)s}
|c_{j}|^{2}-a^{2\tau}\lambda_{j}^{2\tau s}|c_{j}|^{2}),
$$
that gives the desired inequality (4.2) if $t=s\tau.$
\end{proof}

To prove the Theorem 4.2 it is enough to apply the last Lemma 4.3
to the Lemma 4.1 with $ a=C(M)\rho^{2d}$.

Now we can formulate and prove our Approximation Theorem.
\begin{thm}
There exist constants $C(M), \rho(M)>0$ such that for any
$0<\rho<\rho(M)$, any $\rho$-lattice  $M_{\rho}$, any smooth
function $f$ and any $t\geq 0$ the following inequality holds true
$$
\|(1+\Delta)^{t}(s_{2^{m}d+t}(f)-f)\|\leq
\left(C(M)\rho^{2}\right)^{2^{m}d} \|(1+\Delta)^{2^{m}+t}f\|,
$$
for any $m=0, 1, ... .$ In particular, if $f$ is an $\omega$-band
limited function, i.e. $f$ is a linear combination of orthonormal
eigen functions whose corresponding eigen values belong to the
interval $[0, \omega],$ then
$$
\|(1+\Delta)^{t}(s_{2^{m}d+t}(f)-f)\|\leq (1+\omega)^{t}
\left(C(M)\rho^{2}(1+\omega)\right)^{2^{m}d}\|f\|,
$$
where $m=0, 1, ... .$

Moreover, if $t>d/2+k$ then there exists a $C(M,t)$ such that

$$\|(s_{2^{m}d+t}(f)(x)-f(x))\|_{C^{k}(M)}\leq \left(C(M,t)\rho^{2}\right)
^{2^{m}d} \|(1+\Delta)^{2^{m}d+t}f\|, m=0, 1, ...
$$
and respectively,
$$
\|(s_{2^{m}d+t}(f)(x)-f(x))\|_{C^{k}(M)}\leq (1+ \omega)^{t}
\left(C(M,t)\rho^{2}(1+\omega)\right)^{2^{m}d}\|f\|, m=0, 1, ... ,
$$
if $f$ is an $\omega$-band limited function.
\end{thm}

\begin{proof}
To prove first two inequalities it is enough to use Theorem 4.2,
the minimization property of splines and the inequality
$$
\|(1+\Delta)^{t}f\|\leq (1+\omega )^{t}\|f\|
$$ for all $f$ which are linear combinations of eigen functions
whose eigen values are not greater $\omega$.

To prove the inequalities in the uniform norms we use the Sobolev
Embedding Theorem. The Approximation Theorem 4.4 is proved.
\end{proof}

\section{Applications}

\textbf{Example 1}

In this section we illustrate our results in the case of the
hemispherical transform [15].

 We consider the unit sphere
$S^{d}\subset \mathbb{R}^{d+1}$ and the corresponding space
$L_{2}(S^{d})$ constructed with respect to normalized and
rotation-invariant measure. Every vector $\xi\in S^{d}$ defines a
hemisphere $h_{\xi}\subset S^{d}$ as the set of all vectors $x\in
S^{d}$ for which $\xi\cdot x>0$, where $\xi\cdot x$ is the
standard inner product in $\mathbb{R}^{d+1}$. The correspondence
$$
h_{\xi}\rightarrow \xi
$$
will be treated as a correspondence between the set of all
hemispheres of $S^{d}$ and points of the dual sphere $S^{d}_{*}$.

We denote by $Y^{i}_{j}$ an orthonormal basis of spherical
harmonics in the space $L_{2}(S^{d})$, where $j=0,1, ... ; i=
1,2,..., n_{d}(j)$ and
$$
n_{d}(j)=(d+2j-1)\frac{(d+j-2)!}{j!(d-1)!}
$$
is the dimension of the subspace of spherical harmonics of degree
$j$.

The Fourier decomposition of $f\in L_{2}(S^{d})$ is
$$
f=\sum_{i,j}c_{i,j}(f)Y^{i}_{j},
$$
where
$$
c_{i,j}(f)=\int_{S^{d}}f\overline{Y^{i}_{j}}dx=<f,Y_{j}^{i}>_{L_{2}(S^{d})}.
$$

 To every function  $f\in L_{2}(S^{d})$ the hemispherical transform $T$
 assigns a function $Tf\in L_{2}(S_{*}^{d})$ on the dual sphere $S^{d}_{*}$
 which is given
by the formula
$$
(Tf)(\xi)=\int_{\xi\cdot x >0} f(x)dx.
$$

For every function $f\in L_{2}(S^{d})$ that has Fourier
coefficients $c_{i,j}(f)$ the hemispherical transform can be given
explicitly by the formula
$$
Tf(\xi)=\pi^{(d-1)/2}\sum_{i,j}m_{j}c_{i,j}(f)Y^{i}_{j}(\xi),
\xi\in S^{d}_{*},
$$
where $m_{j}=0, $ if $j$ is even and
$$
m_{j}=(-1)^{(j-1)/2}\frac {\Gamma(j/2)}{\Gamma((j+d+1)/2))},
$$
 if $j$ is odd.

The transformation $T$ is one to one on the subspace of odd
functions (i.e. $f(x)=-f(-x))$ of a Sobolev space
$H_{t}^{odd}(S^{d})$ and maps it continuously onto
$H_{t+(d+1)/2}^{odd}(S^{d}_{*})$,
$$
T(H_{t}^{odd})(S^{d})=H_{t+(d+1)/2}^{odd}(S^{d}_{*}).
$$

Let $\{h_{\nu}\}, \nu=1,2,...,N, $ be a finite set of hemispheres
on $S^{d}$. We consider functionals $F_{\nu}$ on $L_{2}(S^{d})$
which are given by
 formulas
 $$
F_{h_{\nu}}=F_{\nu}=\int_{h_{\nu}}fdx.
 $$

We will assume that the set of points $\Xi=\{\xi_{\nu}\}$ on the
dual sphere $S_{*}^{d}$ that corresponds to the set of hemispheres
$h_{\nu}$ is symmetric in the sense that $\Xi=-\Xi$.
  Under this assumption we choose a $t>0$  and an odd function
 $f$ and consider the following variational problem: find a
 function $s_{t}(f)\in H_{t}(S^{d}), t>0$ such that

1) $ F_{\nu}(s_{t}(f))=F_{\nu}(f), \nu=1,2,...,N,$

2) $s_{t}(f)$ minimizes norm $\|(1+\Delta)^{t/2}s_{t}(f)\|$.

Since $\Xi=-\Xi$ and function $f$ is odd, the solution $s_{t}(f)$
will be an odd function.

According to the Theorem 1.1 the Fourier series of $s_{t}(f)$ is
$$
s_{t}(f)=\sum_{i,j}c_{i,j}(s_{t}(f))Y^{i}_{j},
$$
where the Fourier coefficients $c_{i,j}(s_{t}(f))$ of $s_{t}(f)$
are given by formulas
$$
c_{i,j}(s_{t}(f))=<s_{t}(f),Y^{i}_{j}>=(1+\lambda_{i,j})^{-t}\sum_{\nu=1}^{N}
\alpha_{\nu}(s_{t}(f))\int_{h_{\nu}}Y^{i}_{j}dx,
$$
where vector $\alpha(s_{t}(f))$ is the solution of the following
$N\times N$ system
$$
\sum_{\nu=1}^{N}b_{\nu\mu}\alpha_{\nu}(s_{t}(f))=\int_{h_{\mu}}f
dx, \mu=1, 2, ... N,
$$
where
$$
b_{\nu\mu}=\sum_{i,j}(1+\lambda_{i,j})^{-t}
\int_{h_{\nu}}Y^{i}_{j} dx\int_{h_{\mu}}Y^{i}_{j} dx.
$$

This spline provides the optimal approximation to $f$ in the sense
that it is the center of the convex set $Q(F, f, t, K)$ of all
functions $\psi$ from $H_{t}(S^{d})$ that satisfy

\begin{equation}
\int_{h_{\nu}}\psi dx=\int_{h_{\nu}}f dx
\end{equation}
and the inequality
\begin{equation}
\|(1+\Delta)^{t/2}\psi\|\leq K
\end{equation}
for any fixed $K$ that satisfies the inequality
$$
K\geq
\|s_{t}(f)\|_{t}=\sum_{\nu=1}^{N}\alpha_{\nu}(s_{t}(f))\int_{h_{\nu}}f
dx.
 $$

Our results about the hemispherical transform are summarized in
the following theorem.

\begin{thm}
For a given symmetric set $H=\{h_{\nu}\}, \nu=1,2,...,N,$ of
hemispheres $h_{\nu}$, an odd function $f$ and any $t>0$ define
the function $s_{t}(f)$ by the formula
$$
s_{t}(f)=\sum_{i,j}c_{i,j}(s_{t}(f))Y^{i}_{j},
$$
where
$$
c_{i,j}(s_{t}(f))=(1+\lambda_{i,j})^{-t}\sum_{\nu}^{N}\alpha_{\nu}(s_{t}(f))
\int_{h_{\nu}}Y^{i}_{j}dx,
$$
and
$$
\sum_{\nu=1}^{N}b_{\nu\mu}\alpha_{\nu}(s_{t}(f))=v_{\mu},
v_{\mu}=\int_{h_{\mu}}fdx, \mu=1,2,...,N,
$$
$$b_{\nu\mu}=\sum_{i,j}(1+\lambda_{i,j})^{-t}\int_{h_{\nu}}Y^{i}_{j}dx
\int_{h_{\mu}}Y^{i}_{j}dx.
$$

The function $s_{t}(f)$ is odd and it has the following
properties.

\bigskip

1) Integrals of the function $s_{t}(f)$ over hemispheres $h_{\nu}$
have  prescribed values $v_{\nu}$:

$$
\left(Ts_{t}(f)\right)(\xi_{\nu})=\int_{h_{\nu}}s_{t}(f)dx=
v_{\nu}, \nu=1,2,...,N.
$$

2) Among all functions that satisfy (5.1) function $s_{t}(f)$
minimizes
 the Sobolev norm
 $$ \|(1+\Delta)^{t/2}s_{t}(f)\|=
 \left(\sum_{\nu=1}^{N}\alpha_{\nu}(s_{t}(f))v_{\nu}\right)^{1/2}.$$
3) Function $s_{t}(f)$ is the center of the convex set $Q(F,
f,t,K)$ of all functions  $g$  from $H_{t}(S^{d})$ that satisfy
(5.1) and the inequality
\begin{equation}
\|(1+\Delta)^{t/2}g\|\leq K,
\end{equation}
for any fixed $K\geq
\left(\sum_{\nu=1}^{N}\alpha_{\nu}(s_{t}(f))v_{\nu}\right)^{1/2}.$
In other words for any $g\in Q(F,f,t,K)$
$$
\|s_{t}(f)-g\|_{t}\leq\frac{1}{2}diam Q(F, f, t, K).
$$
\end{thm}

 Our next goal is to estimate the rate of convergence of
 $s_{t}(f)$ to $f$ in situations when either the set $\Xi$ gets
 denser or the smoothness $t$ of splines $s_{t}(f)$ gets larger.

Note that we can also interpolate the hemispherical transform $Tf$
on the set $\Xi=\{\xi_{\nu}\}$ by constructing an odd spline
$$\hat{s}_{\tau}=\hat {s}_{\tau}(T(f)), \tau=t+(d+1)/2,$$
 which is
the solution to the following minimization problem:

\bigskip

1) $\hat{s}_{\tau}(\xi_{\nu})=Tf(\xi_{\nu})$

2) $\hat{s}_{\tau}$ minimizes the functional
 $u\rightarrow \|(1+\Delta)^{\tau/2}u\|$.

\bigskip

The Fourier series of $\hat{s}_{\tau}$ is given by the formula
$$
\hat{s}_{\tau}=\sum_{i,j}c_{i,j}(\hat{s}_{\tau})Y^{i}_{j},
$$
where Fourier coefficients $c_{i,j}(\hat{s}_{\tau})$ of
$\hat{s}_{\tau}$ are given by the formulas
$$
c_{i,j}(\hat{s}_{\tau})=<\hat{s}_{\tau},Y^{i}_{j}>=
(1+\lambda_{i,j})^{-\tau}\sum_{\nu=1}^{N}
\alpha_{\nu}(\hat{s}_{\tau})Y^{i}_{j}(\xi_{\nu}),
$$
where $\alpha(\hat{s}_{\tau})$ is the corresponding "jump" vector
which is the solution of the following $N\times N$ system
$$
\sum_{\nu=1}^{N}b_{\nu\mu}\alpha_{\nu}(\hat{s}_{\tau})=(Tf)(\xi_{\mu}),
 \mu=1, 2,... N,
$$
where
$$
b_{\nu\mu}=\sum_{i,j}(1+\lambda_{i,j})^{-\tau}
Y^{i}_{j}(\xi_{\nu})Y^{i}_{j}(\xi_{\mu}).
$$

The spline $\hat{s}_{\tau}$ is the center of the convex set
$\hat{Q} (\delta, Tf,\tau,\hat{K}), \tau=t+(d+1)/2,$ where
$\delta$ is the family of delta functionals
$\{\delta_{\nu}\}=\{\delta_{\xi_{\nu}}\}, \xi_{\nu}\in \Xi,$ and
$$
\hat{K}=K\|T\|_{H_{t}\rightarrow H_{\tau}}.
$$

Since $\hat{Q}(\delta, Tf,\tau, \hat{K})$ is the image of $Q(F,
f,t,K)$ under the linear hemispherical transform $T$, we obtain
that
$$
Ts_{t}(f)=\hat{s}_{\tau}=\hat{s}_{\tau}(Tf).
$$

Applying our Approximation Theorem we obtain the following result
about convergence of interpolants in the case of hemispherical
transform. In this Theorem we use the following parameter $\rho$
$$
\rho=\sup_{\nu}\inf_{\mu} dist(\xi_{\nu},\xi_{\mu}),
\xi_{\nu},\xi_{\mu}\in \Xi, \nu\neq\mu,
$$
as a measure of the density of the set $\Xi$.
\begin{thm}
There exists a constant $C$ such that  for any $m=0,1,..., $ any
$k<2^{m}d-d/2$ and for any odd smooth $f$ we have
$$
\|s_{2^{m}d}(f)-f\|_{k}\leq
(C\rho^{2})^{2^{m}d}\|(1+\Delta)^{\tau/2}Tf\|, m=0,1,...
,\tau=2^{m}d+(d+1)/2,
$$
and if $f$ is an odd spherical harmonic polynomial of (odd) order
$\leq$ $\omega$ then
$$
\|s_{2^{m}d}(f)-f\|_{k}\leq
(C\rho^{2}(1+\omega))^{2^{m}d}(1+\omega)^{(d+1)/2}\|Tf\|,
m=0,1,... .
$$
\end{thm}

The first inequality shows that for any odd smooth function $f$
the interpolants $s_{2^{m}d}(f)$ of a fixed order $2^{m}d,
m=0,1,..., k<2^{m}d-d/2,$ converge to $f$ in the uniform norm
$C^{k}(M)$ as long as the parameter $\rho$ goes to zero, i.e. the
set $\Xi$ on the dual sphere gets denser.

 The second inequality in the Theorem shows,
that interpolants converge to an odd harmonic polynomial of order
$\omega$ for a fixed set of hemispheres $\Xi$ if
$C\rho^{2}(1+\omega)<1$ and as $m$ goes to infinity. This
statement is an analog of the Sampling Theorem for the
hemispherical transform.

\bigskip

\textbf{Example 2}

We discuss  the spherical Radon transform. It associates to a
function $f$ on $S^{d}$ its integrals over great subspheres:
$$
Rf(\theta^{\perp}\cap S^{d})=\int_{\theta^{\perp}\cap S^{d}}fdx,
$$
where $ \theta^{\perp}\cap S^{d}$ is the great subsphere of $S^{d}$
whose plane has normal $\theta$.

If a function $f\in L_{2}(S^{d})$ has Fourier coefficients
$c_{i,j}(f)$ then its Radon Transform is given by the formula
$$
R(f)=\pi^{-1/2}\Gamma((d+1)/2)\sum_{i,j}r_{j}c_{i,j}(f)Y_{j}^{i}.
$$
where $Y_{j}^{i}$ are the spherical harmonic polynomials and
$$
r_{j}=(-1)^{j/2}\Gamma((j+1)/2))/\Gamma((j+d)/2)
$$
if $j$ is even and $r_{j}=0$ if $j$ is odd.
Because the coefficients $r_{j}$ have asymptotics
$(-1)^{j/2}(j/2)^{(1-d)/2}$ as $j$ goes to infinity we have that
$R$ is a continuous operator from the Sobolev space of even functions
$H^{even}_{t}(S^{d})$ onto the space $H^{even}_{t+(d-1)/2}(S^{d})$.
Its inverse is a continuous operator from the space
$H^{even}_{t+(d-1)/2}(S^{d})$ onto the space $H^{even}_{t}(S^{d})$.

Let $\{w_{\nu}\}, \nu=1,2,...,N, $ be a finite set of equatorial
subspheres on $S^{d}$ of codimension one and distributions
$F_{\nu}$ are given by
 formulas
 $$
 F_{\nu}(f)=\int_{w_{\nu}}fdx.
 $$
 By solving corresponding variational problem we can find a
spline $s_{t}(f)\in H_{t}(S^{d})$ such that
$$
F_{\nu}(s_{t}(f))=F_{\nu}(f), \nu=1,2,...,N,
$$
and $s_{t}(f)$ minimizes norm $\|(1+\Delta)^{t/2}s_{t}(f)\|.$

  Because we are interested in even functions
 on $S^{d}$ it is natural to have even splines. So we will assume
 that the set of points $\Xi=\{\xi_{\nu}\}$ on the dual sphere that
 corresponds to the set of subspheres $w_{\nu}$ is even in the sense
 that $\Xi=-\Xi$. It is clear that this assumption will force our splines
 to be even functions.

Our results about the spherical Radon transform are summarized in
the following theorem.

\begin{thm}
For a given symmetric $\rho$- lattice $W=\{w_{\nu}\}$ of
equatorial
 subspheres $\{w_{\nu}\}, \nu=1,2,...,N,$
an even smooth function $f$   and any $t>d/2+k$ define $s_{t}(f)$
by the formula
$$
s_{t}(f)=\sum_{i,j}c_{i,j}(s_{t}(f))Y^{i}_{j},
$$
where
$$
c_{i,j}(s_{t}(f))=(1+\lambda_{i,j})^{-t}\sum_{\nu=1}^{N}\alpha_{\nu}(s_{t}(f))
\int_{w_{\nu}}Y^{i}_{j}dx,
$$
and
$$
\sum_{\nu=1}^{N}b_{\nu\mu}\alpha_{\nu}(s_{t}(f))=v_{\mu},
v_{\mu}=\int_{w_{\mu}}fdx, \mu=1,2,...,N,
$$
$$b_{\nu\mu}=\sum_{i,j}(1+\lambda_{i,j})^{-t}\int_{w_{\nu}}Y^{i}_{j}dx
\int_{w_{\mu}}Y^{i}_{j}dx.
$$

The function $s_{t}(f)$ is even and  has the following properties.

\bigskip

1) Integrals of $s_{t}(f)$ over subspheres $w_{\nu}$ have
prescribed values $v_{\nu}$

$$
\int_{w_{\nu}}s_{t}(f)dx= v_{\nu}, \nu=1,2,...,N;
$$

2) among all functions that satisfy 1) function $s_{t}(f)$
minimizes
 the Sobolev norm
 $$
 \|(1+\Delta)^{t/2}s_{t}(f)\|=\left(\sum_{\nu=1}^{N}\alpha_{\nu}(s_{t}(f))v_{\nu}\right)^{1/2}
 $$

3) function $s_{t}(f)$ is the center of the convex set
$Q(F,f,t,K)$ of all functions $g$ from $H_{t}(S^{d})$ that satisfy
1) and the inequality
$$
\|(1+\Delta)^{t/2}g\|\leq K,
$$
for any fixed $K\geq
\left(\sum_{\nu=1}^{N}\alpha_{\nu}(s_{t}(f))v_{\nu}\right)^{1/2}.$
In other words for any $g\in Q(F, f,t,K)$
$$
\|s_{t}(f)-g\|_{t}\leq\frac{1}{2}diam Q(F, f, t, K).
$$

\end{thm}

An approximation result similar to the Theorem 5.2 can also be
formulated.

Using our Approximation Theorem we obtain the following result
about convergence of interpolants in the case of Radon transform.
In this Theorem we use the following parameter $\rho$
$$
\rho=\sup_{\nu}\inf_{\mu} dist(\xi_{\nu},\xi_{\mu}),
\xi_{\nu},\xi_{\mu}\in \Xi, \nu\neq\mu,
$$
as a measure of the density of the set $\Xi$.
\begin{thm}
There exists a constant $C$ such that  for any $m=0,1,..., $ any
$k<2^{m}d-d/2$ and for any even smooth $f$ we have
$$
\|s_{2^{m}d}(f)-f\|_{k}\leq
(C\rho^{2})^{2^{m}d}\|(1+\Delta)^{\tau/2}Rf\|, m=0,1,...
,\tau=2^{m}d+(d-1)/2,
$$
and if $f$ is an even spherical harmonic polynomial of (even)
order $\leq$ $\omega$ then
$$
\|s_{2^{m}d}(f)-f\|_{k}\leq
(C\rho^{2}(1+\omega))^{2^{m}d}(1+\omega)^{(d+1)/2}\|Rf\|,
m=0,1,... .
$$
\end{thm}

The first inequality shows that for any even smooth function $f$
interpolants $s_{2^{m}d}(f)$ of a fixed order $2^{m}d, m=0,1,...,
k<2^{m}d-d/2,$ converge to $f$ in the uniform norm $C^{k}(M)$ as
long as the parameter $\rho$ goes to zero, i.e. the set $\Xi$ on
the dual sphere gets denser.

 The second inequality in the Theorem shows,
that interpolants converge in the uniform norm $C^{k}(M)$  to an
even harmonic polynomial of order $\omega$ for a fixed set of
subspheres $\Xi$ if $C\rho^{2}(1+\omega)<1$ and $m$ goes to
infinity. This statement is an analog of the sampling Theorem for
the spherical Radon transform.

\bigskip

\makeatletter \renewcommand{\@biblabel}[1]{\hfill#1.}\makeatother

\end{document}